\documentclass[10 pt]{amsart}

\usepackage{amssymb,amsmath,amsfonts,mathrsfs}

\newtheorem{theorem}{Theorem}[section]
\newtheorem{thm}{Theorem}[section]
\newtheorem{lemma}[theorem]{Lemma}

\newtheorem{prop}[theorem]{Proposition}

\theoremstyle{definition}

\newtheorem{example}[theorem]{Example}

\theoremstyle{remark}
\newtheorem{remark}[theorem]{Remark}

\numberwithin{equation}{section}

\newcommand{\vcomp}{C_{c}(V)}
\newcommand{\lw}{{\ell^{2}_{\mu}}(V)}

\newcommand{\delswa}{\Delta_{\theta}}
\newcommand{\hmax}{H_{\max}}

\newcommand{\Dom}{\operatorname{Dom}}

\title[Perturbed bi-Laplacians on infinite graphs]{Self-adjointness of perturbed bi-Laplacians on infinite graphs}
\author{Ognjen Milatovic}
\address{Department of Mathematics
and Statistics \\ University of North Florida \\ Jacksonville, FL
32224 \\ USA.}
\email{omilatov@unf.edu}

\begin{document}
\maketitle

\begin{abstract} We give a sufficient condition
for the essential self-adjointness of a perturbation of the square of the magnetic Laplacian on an infinite weighted graph.
The main result is applicable to graphs whose degree function is not necessarily bounded. The result allows perturbations that are not necessarily bounded from below by a constant.
\end{abstract}

\section{Introduction}
Spectral-theoretic properties of the Laplacian and Schr\"odinger operators on infinite graphs have
been a fruitful topic of research in recent years. In particular, the notion of (essential) self-adjointness has drawn quite
a bit of attention; see, for instance, the papers~\cite{vtt-11-2,vtt-11-3,HKMW, Jor-08, Keller-Lenz-10,Keller-Lenz-09,Milatovic-Truc,Torki-10}.
The most advanced results to date concerning the self-adjointness of (primarily lower semi-bounded) magnetic Schr\"odinger operators (acting
on vector bundles) over infinite (not necessarily locally finite) graphs are contained in the recent paper~\cite{MS-18}.
For further pointers to the literature on the self-adjointness of Laplace/Schr\"odinger operators on graphs,
we direct the reader to~\cite{Keller-15,LSW-16,LSW-17,MS-18}. Additionally, there is a line of investigation concerning Laplace-type operators acting on 1-forms, initiated in~\cite{Masamune-09} and continued in~\cite{AT-15,BGJ-15}.  Furthermore, the author of~\cite{CY-17} investigated the essential self-adjointness of the Laplacian on a 2-simplicial complex. In addition to self-adjoint operators on graphs, the non-self-adjoint ones have also come into the focus of researchers, as seen in the papers~\cite{ABTH-19} and~\cite{ABTH-20}.

In a recent article~\cite{G-Mug}, the authors explored, among other things, semigroup properties and well-posedness of the (linear) parabolic problems corresponding to the square of the Laplacian (also known as the bi-Laplacian) on discrete graphs and metric graphs. As indicated in section 1 of~\cite{G-Mug}, the bi-Laplacian operator on discrete graphs is naturally associated with the discrete variant of the beam equation. Concerning the discrete graph aspect of~\cite{G-Mug}, we should mention that the authors focused on finite graphs and (see remark 2.7 in~\cite{G-Mug}) infinite graphs with vertex weights and edge weights equal to $1$ and with uniformly bounded degree, which leads to a realization of the bi-Laplacian as a bounded self-adjoint operator on the corresponding $L^2$-space.

Working in the setting of (locally finite) infinite graphs, we consider a discrete bi-Laplacian whose self-adjoint realization is (generally) an unbounded operator. More precisely, we study the (essential) self-adjointness of $\delswa^2$, the square of the discrete magnetic Laplacian $\delswa$, perturbed by a real-valued potential $W$. The requirement imposed on the graph in our main result, Theorem~\ref{T:main-3}, incorporates the behavior of the weighted degree and vertex degree over the balls of (combinatorial distance) radius $n\in\mathbb{Z}_{+}$. As far as the perturbation $W$ is concerned, we require that it satisfy
$W(x)\geq -(q\circ r)(x)$, for all vertices $x$ of the graph, where $r(x)$ is the usual combinatorial distance of a vertex $x$ from a reference vertex $x_0$ and $q$ is a non-negative non-decreasing function such that $q(s)=O(s^{\alpha})$ with $0\leq\alpha\leq 1$.
Thus, we allow functions $W$ that are not necessarily bounded from below by a constant. In a certain sense, Theorem~\ref{T:main-3}
can be viewed as a discrete analogue of the self-adjointness result in~\cite{Mi-2017} for the perturbed bi-Laplacian on a geodesically
complete Riemannian manifold with Ricci curvature bounded from below by a (possibly unbounded) non-positive function depending on the distance from a reference point.

We should point out that the definition~(\ref{E:magnetic-lap}) below for $\delswa$ can be combined with $(\delswa^2u)(x):=\delswa((\delswa u)(x))$ to obtain a formula relating $\delswa^2$ and $\delswa$. Naturally, one may wonder if the (essential) self-adjointness of $\delswa^2$ can be deduced by using one of the existing (essential) self-adjointness results for $\delswa$ (or, perhaps, a Schr\"odinger operator $\delswa + W_1$, where $W_1$ is a function on the vertices of the graph). Unfortunately, a cumbersome nature of the relationship between $\delswa^2$ and $\delswa$, as one can already see in the special case $b(x,y)\in\{0,1\}$, $\mu(x)\equiv 1$, and $\theta(x,y)\equiv 0$, where $b$ and $\mu$ are, roughly speaking, the edge and the vertex weights respectively and $\theta$ is a phase function (see section~\ref{SS:setting} for definitions), hinders the application of the self-adjointness results for the discrete Laplacian/Schr\"odinger operator to the situation of a (perturbed) discrete bi-Laplacian.

The methodological core of the paper is located in section~\ref{S:prelim-est}, where key estimates are carried out.
For the purpose of localizing the problem, we borrowed the cut-off functions from the paper~\cite{Masamune-09}.
In contrast to Laplace (or Schr\"odinger) operators, the estimates in our situation are more involved partly because
the underlying quadratic form is ``driven" by the (magnetic) Laplacian, whose product rule produces one more term than
than the product rule for the ``first-order" differential. To make it easier to follow the presentation in section~\ref{S:prelim-est},
we put the two central estimates in Propositions~\ref{P:p-2} and~\ref{P:p-4} and the auxiliary ones in lemmas.

The article consists of six sections. We placed the notations and the statement of the main result in section~\ref{S:main}. The examples illustrating the main theorem are located in section~\ref{S:ex}. Subsequently, in section~\ref{S:mo-sa}, we stated the definition of the maximal operator associated to $\delswa^2 +W$. In the same section, we recalled the abstract self-adjointness criterion used in the proof of the main result. The preliminary estimates are contained in section~\ref{S:prelim-est}, and, lastly, the proof of the main result can be found in section~\ref{S:proof-result}.

\section{Main Results}\label{S:main}  We begin this section by describing the notations used in the rest of the article.
\subsection{The setting}\label{SS:setting}
Throughout the paper, the vertex set $V$ is a countably infinite set equipped with a measure $\mu\colon V\to (0,\infty)$. Additionally, $b\colon V\times V\to[0,\infty)$ is a function satisfying the following conditions:
\begin{enumerate}
  \item [(b1)] $b (x, y) = b (y, x)$, for all $x,\,y\in V$;
  \item [(b2)] $b(x,x)=0$, for all $x\in V$;
  \item [(b3)] $\textrm{deg}(x):=\displaystyle\sharp\,\{y\in V\colon b(x,y)>0\}<\infty$, where $\sharp\, G$ is understood as the number of elements in the set $G$.
\end{enumerate}

We will call two vertices $x,\, y\in V$  \emph{neighbors} and indicate the corresponding relationship by $x\sim y$ whenever the following property is satisfied: $b(x, y) > 0$. We will call the triple $(V,b,\mu)$ a \emph{locally finite weighted graph}, with the label ``locally finite" coming from property (b3). Two neighbors $x\sim y$ give rise to an (oriented) edge, which we will indicate by $e=(x,y)$ or $e=(y,x)$, and the set of (oriented) edges will be denoted by $E$.

In this article, $(V,b,\mu)$ is always considered to be \emph{connected}: any two vertices $x,\,y\in V$ can be joined by a path $\gamma$.
Here, by a path $\gamma$ we mean a finite sequence of vertices $x_0,\,x_1,\,\dots,x_n$  such that $x=x_0$, $y=x_n$,
and $x_{j}\sim x_{j+1}$ for all $0\leq j\leq n-1$. The (combinatorial) length of a path $\gamma$ is defined as the number edges in this path. The (combinatorial) distance $d(x,y)$ between vertices $x$ and $y$ is defined as the (combinatorial) length of the shortest path connecting the vertices $x$ and $y$. For simplicity, we will drop the words ``combinatorial" when referring to the length of a path or distance between vertices, as we will not be using any other length or distance notions in this paper.

Fixing a vertex $x_0\in V$ and $n\in\mathbb{Z}_{+}:=\{1,2,\dots\}$, we define
\begin{equation}\label{E:def-r}
r(x):=d(x_0,x),
\end{equation}
\begin{equation}\nonumber
B(x_0,n):=\{x\in V\colon r(x)\leq n\}\cup \{e\in E\colon e=(x,y),\, r(x)\leq n \textrm{ and }r(y)\leq n\}.
\end{equation}

We now describe the function spaces used in this paper. We start with $C(V)$--the set of complex-valued functions on $V$ and $\vcomp$--the set of finitely supported elements of $C(V)$. The symbol $\lw$ stands for the space of functions $f\in C(V)$ such that
\begin{equation}\label{E:l-p-def}
\|f\|^{2}:=\sum_{x\in V}\mu(x)|f(x)|^2<\infty,
\end{equation}
where $|\cdot|$ denotes the modulus of a complex number. Note that $\lw$ is a Hilbert space with the inner product
\begin{equation}\label{E:inner-w}
(f,g):=\sum_{x\in V}\mu(x)f(x)\overline{g(x)}.
\end{equation}

Before describing the operators used in this paper, we introduce a phase function $\theta \colon V\times V\to [-\pi,\pi]$, such that $\theta(x,y)=-\theta(y,x)$ for all $x\,,y\in V$. To simplify the notation in the sequel, we will write $\theta_{x,y}:=\theta(x,y)$.
We can now define the formal magnetic Laplacian $\delswa\colon C(V)\to C(V)$ on $(V,b,\mu)$ as follows:
\begin{equation}\label{E:magnetic-lap}
    (\delswa u)(x)=\frac{1}{\mu(x)}\sum_{y\in V}b(x,y)(u(x)-e^{i\theta_{x,y}}u(y)).
\end{equation}
If $\theta\equiv 0$ we get as a special case the formal Laplacian, which we denote by $\Delta$.
The square of the formal magnetic Laplacian $\delswa^2$ is called \emph{the magnetic bi-Laplacian} or \emph{magnetic biharmonic operator}.
The central object of study in this paper is the operator $H\colon C(V)\to C(V)$
\begin{equation}\label{E:magnetic-schro}
Hu:= \delswa^2 u +Wu,
\end{equation}
where $W\colon V\to \mathbb{R}$.

\subsection{Statement of the result}
For $n\in\mathbb{Z}_{+}:=\{1,2,\dots\}$, define
\begin{equation}\label{E:deg-n}
d_n:=\max_{x\in B(x_0,n)}(\textrm{deg}(x)),\qquad p_n:=\max_{x\in B(x_0,n)}\left(\max_{y\in V}b(x,y)\right),
\end{equation}
where $\textrm{deg}(x)$ is as in (b3). We are now ready to state the main result.
\begin{thm}\label{T:main-3} Assume that $(V, b, \mu)$ is a locally finite, weighted, and connected graph such that $\mu(x)\geq \mu_0$, where $\mu_0>0$ is some constant. Let $q\colon [0,\infty)\to (0,\infty)$ be a non-decreasing function such that $q(s)=O(s^{\alpha})$, for some $0\leq \alpha\leq 1$.  Let $H$ be as in~(\ref{E:magnetic-schro}) with $W\colon V\to \mathbb{R}$ satisfying
\begin{equation}\label{E:W-minorant-1}
W(x)\geq -q(r(x)),\quad \textrm{for all }x\in V,
\end{equation}
where $r(x)=d(x_0,x)$. In the case $0<\alpha\leq 1$, assume that $\{n^{\alpha-1}d_np_n\}$ is a bounded sequence. In the case $\alpha=0$, assume that there exist numbers $0<K<\frac{\mu_0}{2}$ and $N\in\mathbb{Z}_{+}$ such that
\begin{equation}\label{E:graph-assumption}
\frac{d_np_n}{n}\leq K,
\end{equation}
for all $n\geq N$. Then, $H$ is essentially self-adjoint on $\vcomp$.
\end{thm}
\begin{remark}\label{R:boundedness} In the case $0<\alpha\leq 1$, boundedness of the sequence $\{n^{\alpha-1}d_np_n\}$ implies that $\frac{d_np_n}{n}\to 0$ as $n\to\infty$, which guarantees the fulfilment of~(\ref{E:graph-assumption}) with $0<K<\frac{\mu_0}{2}$.
\end{remark}

\section{Examples}\label{S:ex}
As an illustration of Theorem~\ref{T:main-3}, in this section we provide some examples.
\begin{example}\label{ex-1} We consider a graph with vertices $V=\{0,1,2,\dots\}$ and edges $k\sim k+1$, for all $k\in V$.
In this example, we assign the edge weights $b(k,k+1)=1$ and vertex weights $\mu(k)=1$, for all $k\in V$. Let $n\in\mathbb{Z}_{+}$ and let $B(x_0,n)$ be as in section~\ref{S:main} with $x_0=0$. Referring to~(\ref{E:deg-n}), we see that in this example $d_n=2$, $p_n=1$,   and~(\ref{E:graph-assumption}) is satisfied. Thus, by Theorem~\ref{T:main-3}, the operator $\delswa^2$ is essentially self-adjoint on $C_{c}(V)$. Define a function $q\colon [0,\infty)\to(0,\infty)$  by the formula $q(s):=s+1$,
and a function $W\colon V\to \mathbb{R}$ by the formula $W(k):=-k$. We see that $q(s)$ is non-decreasing and $q(s)=O(s^{\alpha})$ with $\alpha=1$. Moreover, the property~(\ref{E:W-minorant-1}) is satisfied for all $x\in V$.
Finally, we observe that the sequence $\{n^{\alpha-1}d_np_n\}=\{2\}$ is bounded.  Therefore, by Theorem~\ref{T:main-3} the operator $ \delswa^2 u +Wu$ is essentially self-adjoint on $C_{c}(V)$.
\end{example}

\begin{example}\label{ex-2} We consider a graph with the same vertices and edges as in Example~\ref{ex-1}. We assign the edge weights $b(k,k+1)=\sqrt{k+1}$ and vertex weights $\mu(k)=1$, for all $k\in V$. Let $n\in\mathbb{Z}_{+}$ and let $B(x_0,n)$
be as in Example~\ref{ex-1}. We see that in this example $d_n=2$, $p_n=\sqrt{n+1}$,
and~(\ref{E:graph-assumption}) is satisfied. Thus, by Theorem~\ref{T:main-3}, the operator $\delswa^2$ is essentially self-adjoint on $C_{c}(V)$. Define a function $q\colon [0,\infty)\to(0,\infty)$  by the formula $q(s):=\sqrt{s}+1$,
and a function $W\colon V\to \mathbb{R}$ by the formula $W(k):=-\sqrt{k}$. We see that $q(s)$ is non-decreasing and $q(s)=O(s^{\alpha})$ with
$\alpha=\frac{1}{2}$. Moreover, the property~(\ref{E:W-minorant-1}) is satisfied for all $x\in V$.
Finally, we observe that the sequence $\{n^{\alpha-1}d_np_n\}=\{\frac{2\sqrt{n+1}}{\sqrt{n}}\}$ is bounded.
Therefore, by Theorem~\ref{T:main-3} the operator $ \delswa^2 u
+Wu$ is essentially self-adjoint on $C_{c}(V)$.
\end{example}

In the preceding two examples, the degree function $\textrm{deg}(x)$ was bounded. The next example illustrates that Theorem~\ref{T:main-3}
can be applied to graphs which do not necessarily have bounded degree.

\begin{example}\label{ex-3} Consider a radial tree with (countably infinite) vertex set $V$ and origin $x_0$.
 For $n\in\mathbb{Z}_{+}\cup \{0\}$ define
 \[
 S_n:=\{x\in V\colon d(x,x_0)= n\},
 \]
where $d(\cdot,\cdot)$ is as in~(\ref{E:def-r}). For a vertex $x\in V$ define
\[
\mathscr{N}(x):=\{y\in V\colon y\sim x\}.
\]
Assume that our radial tree has the following property: for every $x\in S_n$,
\begin{equation}
\sharp\,\{\mathscr{N}(x)\cap S_{n+1}\}= \lfloor n^{\kappa}\rfloor +1
\end{equation}
where $\kappa\geq 0$ and the notation $\lfloor a\rfloor$ stands for the greatest integer less than or equal to $a$. We assign the edge weights  and vertex weights as follows: $b(x,y)=1$ for all $x\sim y$, and $\mu(x)=1$ for all $x\in V$.
We see that in this example $p_n=1$ and $d_n= \lfloor n^{\kappa}\rfloor +2$. From now on in this example, we assume that $0\leq\kappa< 1$. Then,~(\ref{E:graph-assumption}) is satisfied, and, by Theorem~\ref{T:main-3}, the operator $\delswa^2$ is essentially self-adjoint on $C_{c}(V)$. Let $\alpha$ be a real number such that $0<\alpha\leq 1-\kappa$. Define $W(x):=-n^{\alpha}$
for all $x\in S_{n}$, with $n\in\mathbb{Z}_{+}\cup \{0\}$. Note that~(\ref{E:W-minorant-1}) holds with $q(s)=s^{\alpha}$. From our assumptions on $\alpha$ and $\kappa$, we see that the sequence $\{n^{\alpha-1}d_np_n\}$ is bounded. Therefore, under the stated assumptions on $\alpha$ and $\kappa$, by Theorem~\ref{T:main-3} the operator $ \delswa^2 u+Wu$ is essentially self-adjoint on $C_{c}(V)$.
\end{example}

\section{Maximal Operator and Self-adjointness}\label{S:mo-sa}
Before launching into the technical estimates, we define the maximal operator $\hmax$ as follows:
$\hmax u:=Hu$ for all $u\in\Dom(\hmax):=\{u\in\lw\colon Hu\in\lw\}$,
where $H$ is as in~(\ref{E:magnetic-schro}). It turns out that the following operator equality holds: $(H|_{C_{c}(V)})^{*}=\hmax$,
where the symbol $T^*$ indicates the adjoint of an operator $T$ in $\lw$. Keeping in mind that our graph is locally finite,
the proof of the latter operator equality proceeds in the same way as the one for magnetic Schr\"odinger operators
in Proposition 3.17(a) of~\cite{MS-18}.

By the corollary to abstract theorem X.2 in~\cite{rs}, establishing the essential self-adjointness of $H|_{C_{c}(V)}$, that is, the self-adjointness of the closure $\overline{H|_{C_{c}(V)}}$ in $\lw$ is equivalent to showing that $\delta_{+}=\delta_{-}=0$, where the deficiency indices $\delta_{\pm}$ are defined as follows: $\delta_{\pm}:=\dim\ker((H|_{C_{c}(V)})^{*}\mp \nu i)=\dim\ker(\hmax \mp \nu i)$, where $\nu>0$ is some number. (By abstract theorem X.1 in~\cite{rs}, the numbers  $\delta_{\pm}$ are independent of $\nu>0$.)
\section{Preliminary Estimates}\label{S:prelim-est}
With $\mu$, $b$, and $\theta$ as in~(\ref{E:magnetic-lap}) and for $u\in C(V)$ and $\psi\in C(V)$,
define
\begin{equation}\label{E:P-psi}
(P_{\psi}[u])(x):=\frac{1}{\mu(x)}\sum_{y\in V}b(x,y)(\psi(x)-\psi(y))(u(x)-e^{i\theta_{x,y}}u(y)).
\end{equation}
We begin with the following well-known lemma, whose proof is included for convenience.
\begin{lemma}\label{L:lem-1} Let $\delswa$ and $\Delta$ be as in section~\ref{SS:setting}. Assume that $u\in C(V)$ and $\psi\in C(V)$. Then,
\[
\delswa(\psi u)=\psi \delswa u-P_{\psi}[u]+u\Delta\psi,
\]
where $P_{\psi}[u]$ is as in~(\ref{E:P-psi}).
\end{lemma}
\begin{proof} Starting from the right hand side, for all $x\in V$ we have
\begin{align}\nonumber\\
&\psi(x) (\delswa u)(x)+u(x)(\Delta \psi)(x)-P_{\psi}[u](x)\nonumber\\
&=\frac{\psi(x)}{\mu(x)}\sum_{y\in V}b(x,y)(u(x)-e^{i\theta_{x,y}}u(y))+\frac{u(x)}{\mu(x)}\sum_{y\in V}b(x,y)(\psi(x)-\psi(y))\nonumber\\
&-\frac{1}{\mu(x)}\sum_{y\in V}b(x,y)(\psi(x)-\psi(y))(u(x)-e^{i\theta_{x,y}}u(y))\nonumber\\
&=\frac{1}{\mu(x)}\sum_{y\in V}b(x,y)(\psi(x)u(x)-e^{i\theta_{x,y}}\psi(y)u(y))=(\delswa (\psi u))(x),\nonumber
\end{align}
and this proves the lemma.
\end{proof}
Before stating the preliminary estimates, we define a sequence  of cut-off functions.   Fix  $x_0\in V$ and let $d(x_0,\cdot)$ and $r(\cdot)$ be as in section~\ref{SS:setting}. Define
\begin{equation}\label{E:cut-off}
\chi_n(x):=\left(\left(\frac{2n-d(x_0,x)}{n}\right)\vee 0\right)\wedge 1,\qquad x\in V,\quad n\in \mathbb{Z}_{+},
\end{equation}
where $a\wedge z$ stands for $\min\{a,z\}$ and $a \vee z$ indicates $\max\{a,z\}$. It was shown in~\cite{Masamune-09} that the sequence $\{\chi_n\}_{n\in\mathbb{Z}_{+}}$ has the following properties: (i) $0\leq \chi_n(x)\leq 1$, for all $x\in V$; (ii)  $\chi_n(x)=1$ for all $x\in B(x_0,n)$; (iii) $\chi_n(x)=0$ for all $x\notin B(x_0, 2n)$; (iv) the support of $\chi_n$ is a finite set; (v) for all $x\in V$, we have $\displaystyle\lim_{n\to\infty}\chi_n(x)=1$; (vi) the functions $\chi_n$ satisfy the inequality
\begin{equation}\label{E:diff-cut-off}
|\chi_n(x)-\chi_n(y)|\leq \frac{1}{n},\qquad\textrm{for all }x\sim y;
\end{equation}
and, finally, (vii) if $x\in V$ satisfies $\chi_n(x)\neq 0$, then
\begin{equation}\nonumber
\frac{\chi_n(y)}{\chi_n(x)}\leq 2,  \qquad\textrm{for all }y\sim x.
\end{equation}
and, consequently,
\begin{equation}\label{E:property-chi-3}
\chi_n(x)+\chi_n(y)\leq 3\chi_{n}(x), \qquad\textrm{for all }y\sim x.
\end{equation}

We are now ready for preliminary estimates. To motivate the discussion, let $\chi_n$ be as in~(\ref{E:cut-off}), let $u\in\lw$ be arbitrary, and
let us use Lemma~\ref{L:lem-1} to expand $(\delswa(\chi_n u), \delswa(\chi_n u))$, where $(\cdot,\cdot)$ and $\|\cdot\|$ are as in~(\ref{E:inner-w}) and~(\ref{E:l-p-def}). We have
\begin{align}\nonumber
&(\delswa(\chi_n u), \delswa(\chi_n u))=(\chi_n \delswa u-P_{\chi_n}[u]+
u\Delta\chi_n, \chi_n \delswa u-P_{\chi_n}[u]+u\Delta\chi_n)\nonumber\\
&=\|\chi_n \delswa u\|^2-2\textrm{Re }(\chi_n\delswa u,P_{\chi_n}[u])+2\textrm{Re }(\chi_n\delswa u,u\Delta\chi_n)\nonumber\\
&-2\textrm{Re }(P_{\chi_n}[u], u\Delta\chi_n)+\|P_{\chi_n}[u]\|^2+\|u\Delta\chi_n\|^2,\nonumber
\end{align}
which shows
\begin{align}\label{E:chi-del-u}
&\|\chi_n \delswa u\|^2=\|\delswa(\chi_n u)\|^2+2\textrm{Re }(\chi_n\delswa u,P_{\chi_n}[u])-2\textrm{Re }(\chi_n\delswa u,u\Delta\chi_n)\nonumber\\
&+2\textrm{Re }(P_{\chi_n}[u], u\Delta\chi_n)-\|P_{\chi_n}[u]\|^2-\|u\Delta\chi_n\|^2.
\end{align}
Our first goal is as follows: keeping $\|\delswa(\chi_n u)\|^2$ as is,  estimate (from above) each of the remaining items on the right hand side of~(\ref{E:chi-del-u}) by $\|\chi_n \delswa u\|$ and/or $\|u\|$. Having obtained such individual estimates, the formula~(\ref{E:chi-del-u}) will enable us to estimate $\|\chi_n \delswa u\|$ from above by $\|\delswa(\chi_n u)\|$ and $\|u\|$.

The following notation will be convenient in subsequent discussion. Let $d_n$ and $p_n$ be as in~(\ref{E:deg-n}) and let $\mu_0>0$ be as in the hypothesis of Theorem~\ref{T:main-3}. For $n\in\mathbb{Z}_{+}$ define
\begin{align}\label{E:beta-n}
&\beta_n:=\frac{d_{2n}p_{2n}}{\mu_0 n}.
\end{align}
With our first goal in mind, we begin with the estimates of $\|P_{\chi_n}[u]\|$ and $\|u\Delta\chi_n\|$.
\begin{lemma}\label{E:l-2} Assume that $(V, b, \mu)$ is a locally finite, weighted, and connected graph such that $\mu(x)\geq \mu_0$, where $\mu_0>0$ is some constant.  Let $\beta_n$ be as in~(\ref{E:beta-n}), let $\chi_n$ be as in~(\ref{E:cut-off}), and let $u\in\lw$ be arbitrary. Then,
\begin{align}\label{E:est-P-chi}
&\|P_{\chi_n}[u]\|\leq 2\beta_n\|u\|
\end{align}
\medskip
and
\begin{align}\label{E:est-Delta-chi}
&\|u\Delta \chi_n\|\leq \beta_n\|u\|.
\end{align}
\end{lemma}
\begin{proof}
In the estimates below we will use the following property of real numbers $\{a_j\}_{j=1}^{N}$:
\begin{equation}\label{E:est-N-1}
(a_1+a_2+\dots+a_{N})^2\leq N(a_1^2+a_2^2+\dots a_N^2).
\end{equation}
Starting from~(\ref{E:P-psi}) with $\psi=\chi_n$ and using~(\ref{E:est-N-1}) we have
\begin{align}\nonumber
&\|P_{\chi_n}[u]\|^2=\sum_{x\in V}\left((\mu(x))^{-1}\left|\sum_{y\in V}b(x,y)(\chi_n(x)-\chi_n(y))(u(x)-e^{i\theta_{x,y}}u(y))\right|^2\right)\nonumber\\
&\leq \mu_0^{-1}\sum_{x\in V}\left(\sum_{y\in V} b(x,y)|\chi_n(x)-\chi_n(y)|(|u(x)|+|u(y)|)\right)^2\nonumber\\
&\leq \mu_0^{-1}\sum_{x\in B(x_0,2n)}\textrm{deg}(x)\left(\sum_{y\in B(x_0,2n)} (b(x,y))^2|\chi_n(x)-\chi_n(y)|^2(|u(x)|+|u(y)|)^2\right).\nonumber
\end{align}
Using the definition of $d_{n}$, the property~(\ref{E:diff-cut-off}), and~(\ref{E:est-N-1}) with $N=2$, we continue estimating (followed by some rewriting):
\begin{align}
&\dots\leq 2\mu_0^{-1}n^{-2}d_{2n}\sum_{x\in B(x_0,2n)}\left(\sum_{y\in B(x_0,2n)} (b(x,y))^2(|u(x)|^2+|u(y)|^2)\right)\nonumber\\
&\leq 4\mu_0^{-1}n^{-2}d_{2n}\sum_{x\in B(x_0,2n)}\left(\sum_{y\in B(x_0,2n)} (b(x,y))^2|u(x)|^2\right)\nonumber\\
&= 4\mu_0^{-1}n^{-2}d_{2n}\sum_{x\in B(x_0,2n)}\left(|u(x)|^2\sum_{y\in B(x_0,2n)} (b(x,y))^2\right).\nonumber
\end{align}
To finish the proof of~(\ref{E:est-P-chi}), we use the definitions of $p_{n}$ and $d_n$ and the assumption $\mu(x)\geq \mu_0>0$:
\begin{align}
&\dots\leq 4\mu_0^{-1}n^{-2}d_{2n}\sum_{x\in B(x_0,2n)}|u(x)|^2(p_{2n})^2\textrm{deg}(x)\nonumber\\
&= 4\mu_0^{-1}n^{-2}d_{2n}(p_{2n})^2\sum_{x\in B(x_0,2n)}|u(x)|^2\textrm{deg}(x)\nonumber\\
&\leq 4\mu_0^{-1}n^{-2}(d_{2n}p_{2n})^2\sum_{x\in B(x_0,2n)}|u(x)|^2\nonumber\\ &= 4\mu_0^{-1}n^{-2}(d_{2n}p_{2n})^2\sum_{x\in B(x_0,2n)}(\mu(x))^{-1}\mu(x)|u(x)|^2\nonumber\\
&\leq 4\mu_0^{-2}n^{-2}(d_{2n}p_{2n})^2\|u\|^2=4\beta_n^2\|u\|^2.\nonumber
\end{align}

We now turn to the proof of~(\ref{E:est-Delta-chi}). Using the definition of $\Delta$, the assumption $\mu(x)\geq\mu_0$, and the properties of $\chi_n$, we have
\begin{align}\nonumber
&\|u\Delta \chi_n\|^2=\sum_{x\in V}\left((\mu(x))^{-1}|u(x)|^2\left|\sum_{y\in V}b(x,y)(\chi_n(x)-\chi_n(y))\right|^2\right)\nonumber\\
&\leq \mu_0^{-1}\sum_{x\in V}|u(x)|^2\left(\sum_{y\in V} b(x,y)|\chi_n(x)-\chi_n(y)|\right)^2\nonumber\\
&\leq \mu_0^{-1}\sum_{x\in B(x_0,2n)}\textrm{deg}(x)|u(x)|^2\left(\sum_{y\in B(x_0,2n)} (b(x,y))^2|\chi_n(x)-\chi_n(y)|^2\right)\nonumber\\
&\leq \mu_0^{-1}n^{-2}d_{2n}\sum_{x\in B(x_0,2n)}|u(x)|^2\left(\sum_{y\in B(x_0,2n)} (b(x,y))^2\right)\nonumber\\
&\leq \mu_0^{-1}n^{-2}(d_{2n}p_{2n})^2\sum_{x\in B(x_0,2n)}|u(x)|^2\leq \mu_0^{-2}n^{-2}(d_{2n}p_{2n})^2\|u\|^2=\beta_n^2\|u\|^2,\nonumber
\end{align}
where the steps are justified similarly as in the proof of~(\ref{E:est-P-chi}).
\end{proof}
In the next proposition we accomplish our first goal, which we described below~(\ref{E:chi-del-u}).
\begin{prop}\label{P:p-2} Assume that $(V, b, \mu)$ is a locally finite, weighted, and connected graph such that $\mu(x)\geq \mu_0$, where $\mu_0>0$ is some constant.  Let $\beta_n$ be as in~(\ref{E:beta-n}), let $\chi_n$ be as in~(\ref{E:cut-off}), and let $u\in\lw$ be arbitrary. Then, for all $0<\varepsilon<1$ we have
\begin{equation}\label{E:est-chi-Delta-u}
\|{\chi_n}\delswa u\|^2\leq (1-\varepsilon)^{-1}\|\delswa (\chi_n u)\|^2
+\left(\frac{9+4\varepsilon}{(1-\varepsilon)\varepsilon}\right)\beta^{2}_n\|u\|^2.
\end{equation}
\end{prop}
\begin{proof}
Using Lemma~\ref{E:l-2}, Cauchy-Schwarz inequality, and the property
\begin{equation}\label{E:4-epsilon}
ab\leq \varepsilon a^2+\frac{b^2}{4\varepsilon},
\end{equation}
we estimate (from above) the terms on the right hand side of~(\ref{E:chi-del-u}) to get
\begin{align}\nonumber
&\|\chi_n \delswa u\|^2\leq\|\delswa(\chi_n u)\|^2+2\textrm{Re }(\chi_n\delswa u,P_{\chi_n}[u])\nonumber\\
&-2\textrm{Re }(\chi_n\delswa u,u\Delta\chi_n)+2\textrm{Re }(P_{\chi_n}[u], u\Delta\chi_n)\nonumber\\
&\leq \|\delswa(\chi_n u)\|^2+ 4\beta_n\|\chi_n \delswa u\|\|u\|
+2\beta_n\|\chi_n \delswa u\|\|u\|+4\beta_n^2\|u\|^2\nonumber\\
&=\|\delswa(\chi_n u)\|^2+ 6\beta_n\|\chi_n \delswa u\|\|u\|+4\beta_n^2\|u\|^2\nonumber\\
&\leq \|\delswa(\chi_n u)\|^2+\varepsilon\|\chi_n \delswa u\|^2+9\varepsilon^{-1}\beta^2_n\|u\|^2+4\beta_n^2\|u\|^2\nonumber,
\end{align}
which upon rearranging leads to~(\ref{E:est-chi-Delta-u}).
\end{proof}

Let $\hmax$ be as in section~\ref{S:mo-sa}. To motivate our subsequent estimates, suppose that $u\in\Dom(\hmax)$ and $(H-\lambda i) u=0$ for some $\lambda\in\mathbb{R}$, where $i$ stands for the imaginary unit. With $\chi_n$ be as in in~(\ref{E:cut-off}), we have
\begin{align}\label{E:proof-thm-1}
&i\lambda(u,\chi^2_{n} u)=((\delswa^2+ W)u, \chi^2_{n} u)=(\delswa^2u, \chi^2_{n} u)+ (Wu, \chi^2_{n} u)\nonumber\\
&=(\delswa u, \delswa(\chi^2_{n} u))+(\chi^2_{n}W u,u)=(\chi^2_{n}\delswa u, \delswa u)-(\delswa u, P_{\chi^2_{n}}[u])\nonumber\\
&+(\delswa u, u\Delta \chi^2_{n})+(\chi^2_{n}W u,u).
\end{align}
Taking the real part on both sides and rearranging leads to
\begin{equation}\nonumber
\|\chi_{n}\delswa u\|^2=\textrm{Re }(\delswa u, P_{\chi^2_{n}}[u])
-\textrm{Re }(\delswa u, u\Delta \chi^2_{n})-(\chi^2_{n}W u,u).
\end{equation}
Combining the last equation with~(\ref{E:chi-del-u}) we get
\begin{align}\label{E:del-chi-u}
&\|\delswa(\chi_n u)\|^2=\textrm{Re }(\delswa u, P_{\chi^2_{n}}[u])
-\textrm{Re }(\delswa u, u\Delta \chi^2_{n})-(\chi^2_{n}W u,u)\nonumber\\
&-2\textrm{Re }(\chi_n\delswa u,P_{\chi_n}[u])+2\textrm{Re }(\chi_n\delswa u,u\Delta\chi_n)\nonumber\\
&-2\textrm{Re }(P_{\chi_n}[u], u\Delta\chi_n)+\|P_{\chi_n}[u]\|^2+\|u\Delta\chi_n\|^2.
\end{align}
Our second goal is as follows: after estimating the term $-(\chi^2_{n}W u,u)$ from above by $(\chi_n^2(q\circ r) u,u)$ with the help of~(\ref{E:W-minorant-1}), estimate (from above) the remaining items on the right hand side of~(\ref{E:del-chi-u}) by $\|\chi_n \delswa u\|$ and/or $\|u\|$.  With all individual estimates at our disposal, the formula~(\ref{E:del-chi-u}), Proposition~\ref{P:p-2}, and the assumption~(\ref{E:W-minorant-1}) will enable us to estimate $\|\delswa (\chi_{n}u)\|$ from above by $\|u\|$ and $(\chi_n^2(q\circ r) u,u)$.

Keeping in mind our second goal, we now do two more preliminary estimates.
\begin{lemma} \label{E:l-3} Assume that $(V, b, \mu)$ is a locally finite, weighted, and connected graph such that $\mu(x)\geq \mu_0$, where $\mu_0>0$ is some constant.  Assume that Let $\beta_n$ be as in~(\ref{E:beta-n}), let $\chi_n$ be as in~(\ref{E:cut-off}), and
let $u\in\lw$ be arbitrary. Then,
\begin{equation}\label{E:est-P-chi-2}
|(\delswa u, P_{\chi^2_n}[u])|\leq 6\beta_n\|\chi_{n}\delswa u\|\|u\|
\end{equation}
and
\begin{equation}\label{E:est-Delta-chi-2}
\|(\delswa u,u\Delta \chi^2_n)\|\leq 3\beta_n\|\chi_{n}\delswa u\|\|u\|.
\end{equation}
\end{lemma}
\begin{proof} Referring to~(\ref{E:P-psi}) with $\psi=\chi^2_n$, keeping in mind the hypothesis $\mu(x)\geq\mu_0$, and using the properties of $\chi_n$ we have
\begin{align}\nonumber
&|(\delswa u, P_{\chi^2_n}[u])|=\left|\sum_{x\in V}\delswa u(x)\left(\sum_{y\in V}b(x,y)(\chi^2_n(x)-\chi^2_n(y))(\overline{u(x)}-e^{-i\theta_{x,y}}\overline{u(y)})\right)\right|\nonumber\\
&\leq \sum_{x\in V}|\delswa u(x)|\left(\sum_{y\in V} b(x,y)|\chi_n(x)-\chi_n(y)||\chi_n(x)+\chi_n(y)|(|u(x)|+|u(y)|)\right)\nonumber\\
&\leq 3{n}^{-1}\sum_{x\in B(x_0,2n)}|\delswa u(x)|\left(\sum_{y\in B(x_0,2n)} b(x,y)\chi_n(x)(|u(x)|+|u(y)|)\right)\nonumber\\
&=3{n}^{-1}\sum_{x\in B(x_0,2n)}\chi_n(x)|\delswa u(x)|\left(\sum_{y\in B(x_0,2n)} b(x,y)(|u(x)|+|u(y)|)\right),\nonumber\\
&\leq 3\mu_{0}^{-1/2}{n}^{-1}\sum_{x\in B(x_0,2n)}(\mu(x))^{1/2}\chi_n(x)|\delswa u(x)|\left(\sum_{y\in B(x_0,2n)} b(x,y)(|u(x)|+|u(y)|)\right)\nonumber\\
&\leq 3\mu_{0}^{-1/2}{n}^{-1}\|\chi_n\delswa u\|K_1,\nonumber \end{align}
with
\[
K_1^2:=\sum_{x\in B(x_0,2n)}\left(\sum_{y\in B(x_0,2n)} b(x,y)(|u(x)|+|u(y)|)\right)^2,
\]
where in the second inequality we used~(\ref{E:diff-cut-off}) and~(\ref{E:property-chi-3}), and in the last estimate we applied Cauchy--Schwarz inequality. To finish the proof of~(\ref{E:est-P-chi-2}) we need to estimate $K_1$:
\begin{align}\nonumber
&K_1^2\leq \sum_{x\in B(x_0,2n)}\textrm{deg}(x)\left(\sum_{y\in B(x_0,2n)} (b(x,y))^2(|u(x)|+|u(y)|)^2\right)\nonumber\\
&\leq 2\sum_{x\in B(x_0,2n)}\textrm{deg}(x)\left(\sum_{y\in B(x_0,2n)} (b(x,y))^2(|u(x)|^2+|u(y)|^2)\right)\nonumber\\
&\leq 4d_{2n}\sum_{x\in B(x_0,2n)}\left(\sum_{y\in B(x_0,2n)}(b(x,y))^2|u(x)|^2\right)\nonumber\\
&\leq 4d_{2n}(p_{2n})^2\sum_{x\in B(x_0,2n)}\textrm{deg}(x)|u(x)|^2\leq 4(d_{2n}p_{2n})^2\sum_{x\in B(x_0,2n)}|u(x)|^2\nonumber\\
&\leq 4\mu_0^{-1}(d_{2n}p_{2n})^2\sum_{x\in B(x_0,2n)}\mu(x)|u(x)|^2\leq 4\mu_0^{-1}(d_{2n}p_{2n})^2\|u\|^2,\nonumber
\end{align}
that is,
\[
K_1\leq 2\mu_0^{-1/2}d_{2n}p_{2n}\|u\|,
\]
and this ends the proof of~(\ref{E:est-P-chi-2}).
Turning to the proof of~(\ref{E:est-Delta-chi-2}), we have
\begin{align}\nonumber
&|(\delswa u, u\Delta(\chi^2_n))|=\left|\sum_{x\in V}\delswa u(x)\overline{u(x)}\left(\sum_{y\in V}b(x,y)(\chi^2_n(x)-\chi^2_n(y))\right)\right|\nonumber\\
&\leq \sum_{x\in V}|\delswa u(x)||u(x)|\left(\sum_{y\in V} b(x,y)|\chi_n(x)-\chi_n(y)| |\chi_n(x)+\chi_n(y)|\right)\nonumber\\
&\leq 3{n}^{-1}\sum_{x\in B(x_0,2n)}\chi_{n}(x)|\delswa u(x)||u(x)|\left(\sum_{y\in B(x_0,2n)}b(x,y)\right)\nonumber\\
&\leq 3{n}^{-1}\|\chi_n\delswa u\|K_2\nonumber
\end{align}
with
\[
K_2^2:=\sum_{x\in B(x_0,2n)}(\mu(x))^{-1}|u(x)|^2\left(\sum_{y\in B(x_0,2n)} b(x,y)\right)^2,
\]
where in the second inequality we used~(\ref{E:diff-cut-off}) and~(\ref{E:property-chi-3}),
and in the last estimate we applied Cauchy--Schwarz inequality.
The term $K_2$ is estimated similarly to term $K_1$:
\begin{align}\nonumber
&K_2^2\leq \mu_{0}^{-1}\sum_{x\in B(x_0,2n)}\textrm{deg}(x)|u(x)|^2\left(\sum_{y\in B(x_0,2n)} (b(x,y))^2\right)\nonumber\\
&\leq \mu_0^{-1}d_{2n}(p_{2n})^2\sum_{x\in B(x_0,2n)}\textrm{deg}(x)|u(x)|^2\leq  \mu_0^{-2}(d_{2n}p_{2n})^2\|u\|^2,
\end{align}
that is,
\[
K_2\leq \mu_0^{-1}d_{2n}p_{2n}\|u\|,
\]
and this concludes the proof of~(\ref{E:est-Delta-chi-2}).
\end{proof}

The next proposition accomplishes the goal stated below~(\ref{E:del-chi-u}).

\begin{prop} \label{P:p-4} Assume that $(V, b, \mu)$ is a locally finite, weighted, and connected graph such that $\mu(x)\geq \mu_0$, where
$\mu_0>0$ is some constant.  Let $\beta_n$ be as in~(\ref{E:beta-n}) and let $\chi_n$ be as in~(\ref{E:cut-off}).
Assume that~(\ref{E:W-minorant-1}) is satisfied, with $q$ is as in the hypothesis of Theorem~\ref{T:main-3}. Assume that there exist
numbers $0<C_1<1$ and $N_1\in \mathbb{Z_+}$ such that
\begin{equation} \label{E:cond-beta-n}
\beta_n\leq C_1,
\end{equation}
for all $n\geq N_1$. Additionally, assume that $u\in\Dom(\hmax)$ satisfies $(H-\lambda i) u=0$, for some $\lambda\in\mathbb{R}$.
Then, for all $n\geq N_1$ we have
\begin{equation}\label{E:del-phi-u-final}
\| \delswa (\chi_n u)\|^2\leq C[\beta_n\|u\|^2+((q\circ r)\chi_n u,\chi_n u)],
\end{equation}
where $C$ is a constant independent of $n$, and $(q\circ r)(\cdot)$ is the composition of
$q(\cdot)$ and $r(\cdot)$, with $r(\cdot)$ as in~(\ref{E:def-r}).
\end{prop}
\begin{proof} We showed earlier that~(\ref{E:del-chi-u}) holds for all $u\in\Dom(\hmax)$ such that $(H-\lambda i) u=0$ for some $\lambda\in\mathbb{R}$.
Using Lemma~\ref{E:l-2}, Lemma~\ref{E:l-3}, the assumption~(\ref{E:W-minorant-1}), and Cauchy-Schwarz inequality we estimate (from above) the terms on the right hand side of~(\ref{E:del-chi-u}) to get
\begin{align}
&\|\delswa(\chi_n u)\|^2\leq  6\beta_n\|\chi_{n}\delswa u\|\|u\|+ 3\beta_n\|\chi_{n}\delswa u\|\|u\|+ ((q\circ r)\chi_n u,\chi_n u)\nonumber\\
&+4\beta_n\|\chi_{n}\delswa u\|\|u\|+ 2\beta_n\|\chi_{n}\delswa u\|\|u\| +4\beta^2_n\|u\|^2+ 4\beta^2_n\|u\|^2+\beta^2_n\|u\|^2\nonumber\\
&=15\beta_n\|\chi_{n}\delswa u\|\|u\| +9\beta^2_n\|u\|^2+ ((q\circ r)\chi_n u,\chi_n u)\nonumber\\
&\leq \beta_n\|\chi_{n}\delswa u\|^2+\frac{225\beta_n}{4} \|u\|^2+ 9\beta^2_n\|u\|^2+  ((q\circ r)\chi_n u,\chi_n u),\nonumber
\end{align}
which together with Proposition~\ref{P:p-2} leads to
\begin{align}\label{E:est-1-2-3}
&\|\delswa(\chi_n u)\|^2\leq\beta_n\left((1-\varepsilon)^{-1}\|\delswa (\chi_n u)\|^2
+\left(\frac{9+4\varepsilon}{(1-\varepsilon)\varepsilon}\right)\beta^{2}_n\|u\|^2\right)\nonumber\\
&+\left(\frac{225\beta_n}{4} +9\beta^2_n\right)\|u\|^2+  ((q\circ r)\chi_n u,\chi_n u),
\end{align}
for all $0<\varepsilon<1$.

With $C_1$ as in~(\ref{E:cond-beta-n}), pick $\varepsilon_0$ so that $0<\varepsilon_0<1-C_1$.
Looking at~(\ref{E:cond-beta-n}), we see that $1-\varepsilon_0-\beta_n>0$ for all $n\geq N_1$. This allows us to rearrange~(\ref{E:est-1-2-3})
with $\varepsilon=\varepsilon_0$ as follows:
\begin{align}\nonumber
\|\delswa(\chi_n u)\|^2\leq K(\varepsilon_0,n)\beta_n\|u\|^2+
\left(\frac{1-\varepsilon_0}{1-\varepsilon_0-\beta_n}\right)((q\circ r)\chi_n u,\chi_n u),\nonumber
\end{align}
for all $n\geq N_1$, where
\[
K(\varepsilon_0,n):=\left(\frac{9+4\varepsilon_0}{\varepsilon_0(1-\varepsilon_0-\beta_n)}\right) \beta^2_n+\frac{(225/4+9\beta_n)(1-\varepsilon_0)}{1-\varepsilon_0-\beta_n}.
\]
Finally, looking at $K(\varepsilon_0,n)$ and the coefficient in front of $((q\circ r)\chi_n u,\chi_n u)$ and keeping in mind~(\ref{E:cond-beta-n}) and
\[
\frac{1}{1-\varepsilon_0-\beta_n}\leq \frac{1}{1-\varepsilon_0-C_1},
\]
we see that there exists a constant $C$, independent of $n$, such that~(\ref{E:del-phi-u-final}) holds for all $n\geq N_1$.

\section{Proof of Theorem~\ref{T:main-3}}\label{S:proof-result}
In this section we assume that the hypotheses of Theorem~\ref{T:main-3} are satisfied.  Using the assumption on $q$ and properties of $\chi_n$ we estimate
the second term on the right hand side of~(\ref{E:del-phi-u-final}). Note that there exists $N_2\in\mathbb{Z}_{+}$
such that for all $n\geq N_2$ and all $u\in\lw$, the following holds:
\begin{equation}\label{E:est-q}
|((q\circ r)\chi_n u,\chi_n u)|\leq q(2n)\|u\|^2\leq K_3n^{\alpha}\|u\|^2,
\end{equation}
where $0\leq \alpha\leq 1$, and $K_3>0$ is a constant independent of $n$. Recalling the hypothesis~(\ref{E:graph-assumption}), Remark~\ref{R:boundedness} and the definition~(\ref{E:beta-n}), we see that the condition~(\ref{E:cond-beta-n}) is satisfied for all $n\geq N_1$, where $N_1\in\mathbb{Z}_{+}$. Therefore, we can use Proposition~\ref{P:p-4} in the sequel.

As indicated in section~\ref{S:mo-sa}, to demonstrate the self-adjointness of $\hmax$, it is enough to show that there exists $\nu\in\mathbb{R}-\{0\}$, such that the following property holds:
if $u\in\Dom(\hmax)$ such that  $(H\pm\nu i)u=0$, then $u=0$. Earlier we showed that for $\nu\in\mathbb{R}-\{0\}$ and $u\in\Dom(\hmax)$ such
that  $(H\pm\nu i)u=0$, we have the equality~(\ref{E:proof-thm-1}) with $\lambda$ replaced by $\pm\nu$. Taking the imaginary part on both sides of~(\ref{E:proof-thm-1}) we get
\begin{equation}\label{E:im-part}
\pm\nu\|\chi_{n} u\|^2= -\textrm{Im }(\delswa u, P_{\chi^2_{n}}[u])+ \textrm{Im }(\delswa u, u\Delta \chi^2_{n}).
\end{equation}
Before going further, we record the estimate~(\ref{E:est-chi-Delta-u}) with $\varepsilon=\frac{1}{2}$:
\begin{equation}\label{E:est-eps-1-over-2}
\|{\chi_n}\delswa u\|^2\leq 2\|\delswa (\chi_n u)\|^2+44\beta^{2}_n\|u\|^2.
\end{equation}
We can estimate the terms on the right hand side of~(\ref{E:im-part}) with the help of~(\ref{E:est-P-chi-2}), (\ref{E:est-Delta-chi-2}), (\ref{E:est-eps-1-over-2}) and Proposition~\ref{P:p-4}:
\begin{align}\nonumber
&|\nu|\|\chi_{n} u\|^2\leq 6\beta_n\|\chi_n\delswa u\|\|u\|+3\beta_n\|\chi_n\delswa u\|\|u\|=9\beta_n\|\chi_n\delswa u\|\|u\|\nonumber\\
&\leq  \beta_n\|\chi_n\delswa u\|^2 +\frac{81 \beta_n}{4}\|u\|^2 \leq  2\beta_n\|\delswa (\chi_n u)\|^2
+\left(44\beta^{3}_n+\frac{81 \beta_n}{4}\right)\|u\|^2\nonumber\\
&\leq  2\beta_n \left[C(\beta_n\|u\|^2+((q\circ r)\chi_n u,\chi_n u))\right]+ \left(44\beta^{3}_n+\frac{81 \beta_n}{4}\right)\|u\|^2\nonumber\\
&\leq  2CK_3n^{\alpha}\beta_n \|u\|^2+\left(44\beta^{3}_n+2C\beta^2_n+\frac{81 \beta_n}{4}\right)\|u\|^2,\nonumber
\end{align}
for all $n\geq \max\{N_1,N_2\}$, where in the last inequality we used~(\ref{E:est-q}).
In the case $0<\alpha\leq 1$, the sequence $\{n^{\alpha-1}d_np_n\}$ is bounded (by hypothesis), and,  hence,  the sequence $\{n^{\alpha}\beta_n\}$
is bounded. This implies that $\{\beta_n\}$ is a bounded sequence.  In the case $\alpha=0$, the condition~(\ref{E:cond-beta-n}) is satisfied, and this tells us that the sequence $\{n^{\alpha}\beta_n=\beta_n\}$ is bounded.
Therefore, in both cases, there exists a constant $K_4>0$, independent of $n$, such that
\[
|\nu|\|\chi_{n} u\|^2 \leq K_4\|u\|^2,
\]
for all  $n\geq \max\{N_1,N_2\}$. Letting $n\to\infty$ we get $|\nu|\|u\|^2 \leq K_4\|u\|^2$. For sufficiently large $|\nu|$,
the last inequality leads to $u=0$. This shows that $\hmax$ is self-adjoint, that is, $H$ is essentially self-adjoint on $C_{c}(V)$.

\end{proof}

\end{document}